\documentclass[preprint,11.5pt]{elsarticle}

\usepackage{amsmath,amssymb,amsthm,amscdx}

\usepackage{enumitem}
\usepackage{hyperref}

\numberwithin{equation}{section}

\hypersetup{colorlinks=true}

\renewcommand{\appendixname}{}

\theoremstyle{thmstyleone}
\newtheorem{theorem}{Theorem}[section]
\newtheorem{lemma}[theorem]{Lemma}
\newtheorem{conjecture}[theorem]{Conjecture}

\theoremstyle{thmstyletwo}

\newtheorem{remark}{Remark}

\theoremstyle{thmstylethree}
\newtheorem{definition}{Definition}

\journal{ }

\begin{document}

\begin{frontmatter}

\title{Spline quadrature and semi-classical orthogonal \\ Jacobi polynomials}
\author{Helmut Ruhland}
\ead{Helmut.Ruhland50@web.de}


\begin{abstract}

A theory of spline quadrature rules for arbitrary continuity class in a closed interval $[a, b]$ with arbitrary non-uniform subintervals based on semi-classical orthogonal Jacobi polynomials is proposed.
For continuity class $c \ge 2$ this theory depends on a conjecture.

\end{abstract}


\end{frontmatter}


\section{Introduction}

The problem of optimal quadrature rules for splines has been an active subject of research from the late 1950s to the late 1970s. The past few years have witnessed considerable renewed interest in quadrature rules for splines. This re-emergence has been mainly motivated by their fundamental relevance in the currently popular field of isogeometric analysis. Numerical methods for deriving optimal quadrature rules for spline spaces were proposed in various articles. 

A rich theory of optimal quadrature rules for polynomials in a closed interval, based on the orthogonal Legendre polynomials, exists for two centuries (Gauss-Legendre quadrature rules).

To my knowledge, a theory for $C^c$ polynomial splines with continuity class $c$ in a closed interval $[a, b]$ with arbitrary non-uniform subintervals based on orthogonal polynomials does not exist. Here semi-classical orthogonal Jacobi polynomials are introduced for this purpose.

In this article the method for calculating the nodes and weights of a quadrature rule is: start at the 2 subintervals at the boundaries of the interval $[a, b]$ and parse through all subintervals from the left and right to one in the middle. This method is not new and has been used for special cases in \citet{Nik1996} for uniform $C^1$ cubic splines, in \citet{AitBC2015} for non-uniform $C^1$ cubic splines and in \citet{BAC2017} for uniform $C^1$ quintic splines. It is not necessary to solve a system of equations, the only
non-elementary operation in this process is to get the roots of a polynomial.

For the $C^c$ splines for $c \ge 2$ the application of these orthogonal functions is based on a \textit{conjecture} in section \ref{SectReflMap}. The conjecture is not related to spline quadrature, it is \textit{only} related to rational maps defined by semi-classical orthogonal Jacobi polynomials, so the conjecture is \textit{isolated completely} from spline quadrature.

But there remain and arise more questions, the question of positivity of the weights, the question if the roots of the introduced polynomials are real and if they lie in the interval $[-1, +1]$. Also questions concerning the combinatorial distribution of the nodes in the subintervals are not addressed in this article. So throughout this article under quadrature rules we generally understand rules without the condition of positivity or other conditions on the roots. \newline

The goal of this article is to present the ideas which allow us to use semi-classical Jacobi polynomials in the spline quadrature, assuming the conjecture is true. It is not a mathematically rigorous proof. Proofs done with the assistance of a computer algebra system, which confirm the conjecture for low continuity classes, are not included. \newline

This article is organized as follows. In section \ref{SectDefJacobi}, the two types of semi-classical orthogonal Jacobi polynomials necessary for this article are defined.
In section \ref{SectDefectQM}, for the two introduced types of semi-classical orthogonal Jacobi polynomials the so-called defect of a polynomial is calculated. This is the difference between a weighted sum and the integral over this polynomial. Then the in section \ref{SectDefJacobi} introduced deformation of the Jacobi polynomials by Dirac $\delta$'s allows to express this defect in a simple manner.

The calculation of a defect for the type $Q \, ()$ semi-classical Jacobi polynomials with the formulae of previous section is only possible for the right half of a spline with a support of $2$ intervals. Therefore in section \ref{SectReflMap} a reflection map is introduced that allows to calculate the defect also for the left half of a support $2$ spline in section \ref{SectRecMap}. The conjecture appears in this section.

In section \ref{SectRecMap} the rational recursion map is defined. This allows to proceed from one subinterval to the next. Finally, in section \ref{SectStretch}, to handle the general case of non-uniform subintervals, stretching and the assigned stretching map are defined.

Finally in section \ref{SectFields} the fields for the node positions and weights are presented. This is
done also for the in this article not treated "1/2-rules". An upcoming article will treat these rules. \newline

The formulae in the appendices \ref{c0Formulae} and \ref{c1Formulae} are not used in this article. They are just added to show how these formulae look like in the case of continuity classes $c = 0, 1$ when the conjecture is true. Finally some examples of (sub)optimal quadrature rules obtained with these formulae
are given.

\section{Definition of the 2 types of semi-classical Jacobi orthogonal polynomials \label{SectDefJacobi} }

In the literature the word "semi-classical" orthogonal polynomials is used for polynomials obtained by modifying the weight function of the classical orthogonal polynomials in a certain not specified manner. The weight function $e^{- x^2 - a x^4}$ e.g. results in one kind of semi-classical Hermite orthogonal polynomials $H_n (a, x)$, which for $a=0$ is the classical Hermite polynomial. For the semi-classical orthogonal polynomials defined in here, we modify the weight function of the Jacobi polynomials by adding Dirac $\delta$'s and its derivations at one or at the two ends of a subinterval. This gives us exactly the necessary degrees of freedom for these polynomials.  \newline

Let $c$ be the continuity class of the polynomial splines throughout the whole article. To make the formulae not too overloaded I omit indices in $c$. I write e.g. $Q_n (\textbf{l}, x)$ instead of $Q_n^{(c)} (\textbf{l}, x)$, $J_n (x)$ instead of $P_n^{(c+1,c+1)}(x)$ ... \newline

\begin{definition} \label{DefinitionWeights}
Define  the vector space $\mathbb{D}=\mathbb{R}^{c+1}$, the symbol $\mathbb{D}$ like D(irac) or (D)istribution. This is the space of coefficients of the Dirac $\delta$'s and its derivations in the following weight functions.
We have $\textbf{l} = (l_0, l_1, \dots , l_{c-1}, l_{c}) \in \mathbb{D}$, the bold $\textbf{l}$ like (l)eft when the Dirac's are located at the left side of the interval $[-1, +1]$  at $-1$. In the same manner  $\textbf{r}$ like (r)ight when the Dirac's are located at $+1$.
To define orthogonal polynomials their weight functions are defined:
\begin{equation} W_Q(\textbf{l}, x) = w_Q(x) \left( 1 + \sum_{i=0}^{c} l_i \delta^{(i)}(x+1) \right) \label{WeightQ} \end{equation}
$$ w_Q(x) = (1 - x)^{c+1} $$
\begin{equation} W_M(\textbf{l}, \textbf{r}, x) = w_M(x) \left(1 + \sum_{i=0}^{c} l_i \delta^{(i)}(x+1) + \sum_{i=0}^{c} (-1)^i r_i \delta^{(i)}(x-1) \right) \label{WeightM} \end{equation}
$$ w_M(x) = 1 $$
For $\textbf{l} = \textbf{0}$ the weight function $W_Q(\textbf{0}, x) = w_Q(x)$ is the weight function of the Jacobi polynomials $P_n^{(c+1,0)}(x)$,  for $\textbf{l}, \textbf{r} = \textbf{0}$ the weight function $W_M(\textbf{0}, \textbf{0}, x) = w_M(x) = 1$  is the weight function of the Jacobi polynomials $P_n^{(0,0)}(x)$, the Legendre polynomials $P_n (x)$.
\end{definition}

\begin{definition}
With the weights from the definition above the semi-classical orthogonal Jacobi polynomials $Q_n (\textbf{l}, x)$ and $M_n (\textbf{l}, \textbf{r}, x)$ are defined by these properties:
\begin{align}
\enspace & \int\limits_{-1}^{+1} W_Q(\textbf{l}, x) Q_m (\textbf{l}, x) Q_n (\textbf{l}, x) dx & = 0 \quad
  \hbox{for} \enspace m \ne n \\
\enspace & \int\limits_{-1}^{+1} W_M (\textbf{l}, \textbf{r}, x) M_m (\textbf{l}, \textbf{r}, x) M_n (\textbf{l}, \textbf{r}, x) dx & = 0  \quad  \hbox{for} \enspace m \ne n \nonumber
\end{align}
These orthogonal polynomials are normalized i.e.
$Q_n (\textbf{0}, x) = P_n^{(c+1,0)}(x)$ and $M_n (\textbf{0}, \textbf{0}, x) = P_n^{(0,0)}(x) = P_n (x)$, the Jacobi polynomials. These polynomials don't have factors like $(1 + a l_1 + b l_0 r_1 + \dots )$
i.e. factors independent from $x$. Else such factors would give us normalization for $\textbf{l}$, $\textbf{r} = 0$ too.
For the continuity classes $c = 0, 1$ these semi-classical orthogonal Jacobi polynomials
were already introduced by:
\begin{tabbing}
\= - \citet{Krall1940} \qquad \qquad \qquad  \= with a Dirac $\delta$ at only 1 interval end  \\
\> - \citet{Koor1984}                        \> with 2 Dirac $\delta$'s at both interval ends \\
\> - \citet{ArMaAl2002}                      \> with additional derivations of Dirac's
\end{tabbing}

In newer publications \citep[see, e.g.][]{Dur2022}, see formulae (1.4) ... (1.6) the here defined semi-classical orthogonal polynomials belong to the Krall-Jacobi families.
\end{definition}

With respect to the weights (\ref{WeightQ}) and (\ref{WeightM}) the non-zero scalar products of the corresponding orthogonal polynomials are defined as:
\begin{align}
S_Q (m, \textbf{l}, x) & = \int\limits_{-1}^{+1} W_Q(\textbf{l}, x) Q_m (\textbf{l}, x) Q_m (\textbf{l}, x) dx \label{ScalProdQ} \\
S_M (m, \textbf{l}, \textbf{r}, x)  & = \int\limits_{-1}^{+1} W_M (\textbf{l}, \textbf{r}, x) M_m (\textbf{l}, \textbf{r}, x) M_m (\textbf{l}, \textbf{r}, x) dx \label{ScalProdM}
\end{align}

\section{The fundamental theorem for Gaussian quadrature applied to the semi-classical orthogonal Jacobi polynomials \label{SectDefectQM} }

\begin{theorem}[\cite{Ga2004}, pp. 22-32] \label{FundTheorem}
When the zeros $x_i$ of an orthogonal polynomial $P_n (x)$ with respect to the weight function $\omega (x)$
are taken as nodes  (how to determine the weights $w_i$ see the FT) of a Gaussian
quadrature rule $(x_i, w_i)$ the rule is exact for $deg (f) \le 2n - 1$.
$$ \int\limits_{-1}^{+1} \omega (x) f (x) dx = \sum_{i=1}^{n} w_i f (x_i)$$
\end{theorem}

Looking at the proof of this fundamental theorem it can be seen, that it is also valid when the weight function is a distribution. Our two weight functions in definition \ref{DefinitionWeights}
are distributions.

\begin{remark}
If in the Theorem \ref{FundTheorem} we take the polynomial $P_n (x) + \omega P_{n-1} (x)$ instead of 
$P_n (x)$ the result is a suboptimal quadrature rule depending on $\omega$  that is exact for
$deg (f) \le 2n - 2$.
\end{remark}

\begin{definition}
For given nodes and weigths $(x_i, w_i)$ in an interval or subinterval the defect $D (g)$ for a function
$g (x)$ is the difference between the weighted sum and the integral over $g (x)$:
\begin{equation} D (g) = \sum_{i=1}^{n} w_i \, g (x_i) - \int\limits_{-1}^{+1} g (x) dx
\nonumber \end{equation}
\end{definition}

\subsection{The fundamental theorem applied to the orthogonal, one-sided \textit{Q} () }

\begin{lemma} \label{LemmaFundQ}
Let the $n$ nodes $x_i$ be the roots of $Q_n (\textbf{l}, x)$ and the $n$ weights be:  
\begin{equation} w_i = \frac{a_{n}}{a_{n-1}} \frac{ S_Q (n - 1, \textbf{l}, x_i) }{ Q_n' (\textbf{l}, x_i) \, Q_{n-1} (\textbf{l}, x_i) (1 - x_i)^{c+1}} \label{Weight_iQ} \end{equation}
notice the additional factor $(1 - x_i)^{c+1}$ in the denominator, then the defect for a polynomial $g (x)$
is
\begin{equation}
D (g) = \sum_{i=0}^{c}  (-1)^i l_i \, g^{(i)} (-1) \label{DefectSplineQ}
\end{equation}
under these conditions:
\begin{equation} (1 - x)^{c+1} \mid g (x), \quad deg (g) \le 2n + c  \label{DefectCondQ}  \end{equation}
The first part in these conditions is equivalent to $g^{(k)} (+1) = 0$ for $0 \le k \le c$, this means
$g (x)$ has to be e.g. a spline with a support of $1$ interval or the right half of a spline with a support of $2$ intervals. When $g (x)$ is a spline with a support of $1$ interval then the first part of the conditions
(\ref{DefectCondQ}) is fulfilled and the rhs of the defect is $0$. Then $(x_i, w_i)$ is an exact quadrature rule for splines $S (x)$ with a support of 1 interval.
\end{lemma}

\begin{proof}
Applying the fundamental theorem to the orthogonal polynomials $Q_n (\textbf{l}, x)$ and the
corresponding weight function  $W_Q (\textbf{l}, x)$ we get the following result:
The $n$ nodes $x_i$ are the roots of $Q_n (\textbf{l}, x)$, the weights are:  
$$ w_i^{*} = \frac{a_{n}}{a_{n-1}} \frac{ S_Q (n - 1, \textbf{l}, x_i) }{ Q_n' (\textbf{l}, x_i) \, Q_{n-1} (\textbf{l}, x_i)} $$
the scalar product $S_Q (n, \textbf{l}, x)$ see (\ref{ScalProdQ}), $a_n$ is the coefficient of $x^n$ in $Q_n (\textbf{l}, x)$ 

The equation connecting integral and the weighted sum is:
$$ \sum_{i=1}^{n} w_i^{*} h (x_i) = \int\limits_{-1}^{+1} W_Q (\textbf{l}, x) h (x) dx$$
for $deg (h) \le 2n - 1$

We set now $h (x) = g (x) / w_Q (x) = g (x) / (1 - x)^{c+1}$ and until the rest of this proof we assume the following condition is fulfilled:
\begin{equation} (1 - x)^{c+1} \mid g (x), \quad deg (g) \le 2n + c  \nonumber  \end{equation}

With this new $g (x)$ we get: 
$$ \sum_{i=1}^{n} w_i^{*} / (1 - x_i)^{c+1} g (x_i) = \int\limits_{-1}^{+1} \left( 1 + \sum_{l=0}^{c} l_i \delta^{(i)}(x+1) \right) g (x) dx$$

We define a new weight by $w_i = w_i^{*} / (1 - x_i)^{c+1}$ and get our final formula for the weights:
\begin{equation} w_i = \frac{a_{n}}{a_{n-1}} \frac{ S_Q (n - 1, \textbf{l}, x_i) }{ Q_n' (\textbf{l}, x_i) \, Q_{n-1} (\textbf{l}, x_i) (1 - x_i)^{c+1}} \nonumber \end{equation}

Now with this new weights $w_i$ and with the following definition for the Dirac $\delta$'s
\begin{equation} \int\limits_{}^{} \delta^{(k)} (x + a) f (x) dx = (-1)^k f^{(k)} (-a) \label{DefDirac} \end{equation}
$f^{(k)} (x)$ being the k-th derivation of $f (x)$, and $f^{(0)} (x) = f (x)$,  we get the defect:
\begin{equation} D (g) = \sum_{i=1}^{n} w_i \, g (x_i) - \int\limits_{-1}^{+1} g (x) dx = \sum_{i=0}^{c}  (-1)^i l_i \, g^{(i)} (-1) \nonumber \end{equation}
\end{proof}

\subsection{The fundamental theorem applied to the orthogonal, two-sided \textit{M} () }

\begin{lemma} \label{LemmaFundM}
Let the $n$ nodes $x_i$ be the roots of $M_n (\textbf{l}, \textbf{r}, x)$ and the $n$ weights be:  
\begin{equation}  w_i = \frac{a_{n}}{a_{n-1}} \frac{ S_M (n - 1, \textbf{l}, \textbf{r}, x_i) }{ M_n' (\textbf{l}, \textbf{r}, x_i) \, M_{n-1} (\textbf{l}, \textbf{r}, x_i)} \label{Weight_iM} \end{equation} 
then the defect for a polynomial $g (x)$ is:
\begin{equation}
D (g) = \sum_{i=0}^{c}  (-1)^i l_i \, g^{(i)} (-1)  + \sum_{i=0}^{c} r_i \, g^{(i)} (+1) \label{DefectSplineM}
\end{equation}
under this condition:
\begin{equation} deg (g) \le 2n - 1  \label{DefectCondM}  \end{equation}
When $g (x)$ is a spline with a support of $1$ interval then the rhs of the defect is $0$. Then $(x_i, w_i)$ is an exact quadrature rule for splines $S (x)$ with a support of 1 interval.
\end{lemma}

\begin{proof}
Applying the fundamental theorem to the orthogonal polynomials $M_n (\textbf{l}, \textbf{r}, x)$ and the
corresponding weight function  $W_M (\textbf{l}, \textbf{r}, x)$ we get the following result:

The $n$ nodes $x_i$ are the roots of $M_n (\textbf{r}, x)$, the weights are:  
\begin{equation}  w_i = \frac{a_{n}}{a_{n-1}} \frac{ S_M (n - 1, \textbf{l}, \textbf{r}, x_i) }{ M_n' (\textbf{l}, \textbf{r}, x_i) \, M_{n-1} (\textbf{l}, \textbf{r}, x_i)} \nonumber \end{equation} 
the scalar product $S_M (n, \textbf{l}, \textbf{r}, x)$ see (\ref{ScalProdM}), $a_n$ is the coefficient of $x^n$ in $M_n (\textbf{l}, \textbf{r}, x)$ 

The equation connecting integral and the weighted sum is:
$$  \sum_{i=1}^{n} w_i g (x_i) = \int\limits_{-1}^{+1} W_M (\textbf{l}, \textbf{r}, x) g (x) dx$$
for $deg (g) \le 2n - 1$

$$ \sum_{i=1}^{n} w_i g (x_i) = \int\limits_{-1}^{+1} 
\left( 1 + \sum_{l=0}^{c} l_i \delta^{(i)}(x+1) + \sum_{l=0}^{c} (-1)^i r_i \delta^{(i)}(x-1) \right) g (x) dx$$

Now with the definition (\ref{DefDirac}) of the Dirac $\delta$'s we get the defect:
\begin{equation} D (g) = \sum_{i=1}^{n} w_i \, g (x_i) - \int\limits_{-1}^{+1} g (x) dx = \sum_{i=0}^{c} (-1)^i l_i \, g^{(i)} (-1) + \sum_{i=0}^{c} r_i \, g^{(i)} (+1) \nonumber \end{equation}
\end{proof}

\section{ The reflected one-sided \textit{Q} (), the reflection map and a conjecture \label{SectReflMap} }

With the formula (\ref{DefectSplineM}) and a restriction only in the degree of a polynomial $g (x)$ we can calculate the defect of the polynomial for a subinterval of type $M ()$.

For a subinterval of type $Q \, ()$ the situation is different. The condition (\ref{DefectCondQ}) in the formula (\ref{DefectSplineQ}) for the defect requires $g^{(k)} (+1) = 0$ for $0 \le k \le c$. But we need the defect for polynomials with $g^{(k)} (+1) \ne 0$ and now $g^{(k)} (-1) = 0$ too. This is necessary to calculate the defect of all splines with a support of $2$ intervals. 

Therefore a goal of this section is, to find for the polynomial $Q_n (\textbf{l},\textbf{-} x)$ another vector $\textbf{l}_R$ (the subscript R means reflected) with the property:
$$ Q_n (\textbf{l},\textbf{-} x) = \alpha \, Q_n (\textbf{l}_R, x) $$
The factor $\alpha$ is allowed (and has to be allowed) because we are only interested in the roots of the $Q_n (\textbf{l}, x)$ and a factor doesn't change anything.

\begin{definition} \label{DefinitionSpaceJ}
In \ref{DefinitionWeights} the vector space $\mathbb{D}=\mathbb{R}^{c+1}$  was defined. This is the space
of coefficients of the Dirac $\delta$'s in a weight function for orthogonal polynomials.

Now the projective space $\mathbb{J}=\mathbb{RP}^{c+1}$ ist defined, the symbol $\mathbb{J}$ like J(acobi), the space of coefficients in a sum of symmetric Jacobi polynomials. So we have $\textbf{j} = (j_0 : j_1 : \dots : j_{c} : j_{c+1}) \in \mathbb{J}$. The projective space is the appropriate space, because the roots of the so defined polynomial do not depend an a factor and because the weights (\ref{Weight_iQ}) are homogeneous of degree 0 in this polynomial the whole quadrature rule is invariant under this scaling.
The spaces $\mathbb{D}$ and $\mathbb{J}$ have the same dimension $c+1$.

The following monomial map $r$, the symbol $r$ like r(eflection), is an involution in $\mathbb{J}$, the roots of the corresponding polynomials are just reflected at $0$. The coefficients of the odd degree symmetric Jacobi polynomials are mapped to the negative coefficients (the signs here for even $c$). 
\begin{align}
r : \enspace & \mathbb{J} & \rightarrow & \enspace \mathbb{J} \nonumber \\
& ( j_{0} : j_{1} : \dots : j_{c} : j_{c+1} ) & \mapsto & \enspace ( + j_{0} : - j_{1} : \dots : + j_{c} : - j_{c+1} ) \label{ReflMap} \\
& j_i   & \mapsto & \enspace (-1)^i \, j_i \nonumber
\end{align}
In the following lower case letters are used for maps in $\mathbb{J}$ and the corresponding upper case letters for maps in $\mathbb{D}$
\end{definition}

\begin{conjecture} \label{ConjMain}
The semi-classical Jacobi polynomials $Q_n (\textbf{l}, x)$ can be expressed by a sum of symmetric Jacobi
polynomials $ J_{n} (x) = P_n^{(c+1,c+1)} (x)$, the coefficients $j_{n,i} (\textbf{l})$ are polynomials in the coefficients of the vector $\textbf{l}$ and $n$:
\begin{equation} Q_n (\textbf{l}, x) = \sum_{i=0}^{c+1} j_{n,i} (\textbf{l}) J_{n-i} (x) \label{QasSumJ} \end{equation}
Though not used in this section, to be complete (with different coefficients $j_{n,i}$ now depending on 2 parameters $\textbf{l}, \textbf{r}$):
\begin{equation} M_n (\textbf{l}, \textbf{r}, x) = \sum_{i=0}^{2c+3} j_{n,i} (\textbf{l}, \textbf{r}) J_{n-i} (x) \nonumber \end{equation}
The $j_{n,i} (\textbf{l})$ in (\ref{QasSumJ}) define a polynomial map:
\begin{align}
f_n : \enspace  & \mathbb{D} & \rightarrow & \enspace \mathbb{J} \nonumber \\
& \textbf{l} = ( l_0, \dots , l_{c} ) & \mapsto & \enspace ( j_{n,0} (\textbf{l}) : \dots : j_{n,c+1} (\textbf{l}) ) \label{MapTrans}
\end{align}
The inverse of the map $f$ is defined and the $l_{n,i}$ are rational in $\textbf{j}$: 
\begin{align}
f_n^{-1} : \enspace & \mathbb{J} & \rightarrow & \enspace \mathbb{D} \nonumber \\
& \textbf{j} = ( j_{0} : \dots : j_{c+1} ) & \mapsto & \enspace ( l_{n,0} (\textbf{j}), \dots , l_{n,c} (\textbf{j}) )  \label{MapTransInv}
\end{align}
\end{conjecture}

Assuming this conjecture is true we get the following lemma, which answers the question of this section:
\begin{lemma}
Let $R_n$ be the reflection map $r$ in $\mathbb{J}$ from definition \ref{DefinitionSpaceJ}, transformed by the map $f_n$ and its inverse into an involution in $\mathbb{D}$:
\begin{equation} R_n  = f_n^{-1} \circ r \circ f_n \label{ReflMapR} \end{equation}
where $f \circ g$ is the composition of maps $f \, (g \, ())$, because of the associativity the parenthesis can be omitted. The composition of the maps in (\ref{ReflMapR}) can be visualized by this commutative diagram:
\begin{equation} \begin{CD}
 \mathbb{J}   @>r>>      \mathbb{J}      \\
 @AA{f_n}A               @VV{f_n^{-1}}V  \\
 \mathbb{D}   @>>R_n>    \mathbb{D}
\end{CD} \end{equation}

This $\textbf{l}_R = R_n (\textbf{l})$ has the desired property $ Q_n (\textbf{l},\textbf{-} x) = \alpha \, Q_n (\textbf{l}_R, x) $, i.e. both polynomials have the same roots.
\end{lemma}

\begin{conjecture}
The homogenized rational reflection map $R_n$ and the rational connection map $c_n$ defined in the the next section are birational involutions and so elements of the Cremona Group $Bir \left( \mathbb{P}_\mathbb{C}^n \right)$, for Cremona groups see e.g. \citet{Des2012}. \\
$R_n$ has has degree $n^2$ and the following form, with $\Gamma_n$ and $N_{n,i}$ polynomials in $n$ and $\textbf{l}$, $\Gamma_n$ raised to the powers $c+1, \dots, 1$ in the denominators:
\begin{equation} R_{n,i} : l_i \mapsto N_{n,i} (\textbf{l}) \, / \, {\Gamma_n^{c+1-i} (\textbf{l})} \qquad \hbox{for} \quad 0 \le i \le n - 1 \label{ReflMapForm}
\end{equation}
The rational connection map $c_n$ is the unique involution of degree $n$ (up to a linear transformation) with $detjac$ the union of $n+1$ lines, each line of multiplicity $n-1$. It has the following form:
\begin{equation} \sigma_{n,i} : x_i \mapsto \prod\limits_{\substack{k=0 \\ k \ne i}}^n x_k \qquad \hbox{for} \quad 0 \le i \le n \label{ConnMapTransForm}
\end{equation}
For $n = 2$ this is the so called Veronese map $\sigma_2 : (x_0 : x_1 : x_2) \dashrightarrow (x_1 x_2 : x_0 x_2 : x_0 x_1)$. For higher $n$ this map $\sigma_n$ can be considered as a higher dimensional generalized Veronese map. 
\end{conjecture}

The current state of the conjecture \ref{ConjMain} can be found in appendix \ref{ConjState}. For a connection between the conjectured reflection map $R_n$ and Cremona groups see also appendix \ref{ConCremona}.

\section{ Connecting two subintervals of type \textit{Q}, the recursion map \label{SectRecMap} }

Here a so called recursion map $Rec_n$ is defined. This is a map that transforms the vector \textbf{l} of the subinterval $s$ to the vector $Rec_n \, (\textbf{l})$ of the next one $s+1$, so that the quadrature rule is \textit{exact} for all spline functions with a support of these $2$ intervals. So the
overall quadrature rule is exact for \textit{all} splines (with an arbitray support).

\begin{definition}
The following monomial involution $C$ like (C)onnection is defined, this name because it makes that the sum of the defects in a subinterval $s$ and the following $s+1$ is $0$ and the quadrature rule is exact.
\begin{align}
C : \enspace  & \mathbb{D} & \rightarrow & \enspace \mathbb{D} \nonumber \\
& ( l_0, l_1, \dots , l_{c-1}, l_{c} ) & \mapsto & \enspace ( - l_0, + l_1, \dots , - l_{c-1}, + l_{c} ) \label{ConnMap} \\
& l_i   & \mapsto & \enspace (-1)^{i+1} \, l_i \nonumber
\end{align}
With this connection map and the reflection map (\ref{ReflMapR}) we define the recursion map:
\begin{equation} Rec_n  = C \circ R_n = C \circ ( f_n^{-1} \circ r \circ f_n )
\label{RecMap} \end{equation}
This composition of the maps in (\ref{RecMap}) can be visualized by this diagram:
\begin{equation}
\underbrace
{ \begin{CD}
  \mathbb{J}   @>r>>    \mathbb{J}             @.            \\
  @A{f_n}AA             @VV{f_n^{-1}}V         @.            \\
  \mathbb{D}   @.       \mathbb{D}    @>>C>    \mathbb{D}    \nonumber
  \end{CD}
}_{\quad \enspace Rec_n \enspace : \enspace \mathbb{D} \enspace \longrightarrow \enspace \mathbb{D}}
\end{equation}
This recursion map can also be given in the $\mathbb{J}$ domain now named $rec_n$ (with lower case initial letter):
\begin{equation} rec_n  = c_n \circ r = ( f_n \circ C \circ f_n^{-1} ) \circ r \nonumber \end{equation}
\end{definition}

\begin{theorem} \label{TheoremRecMap}
For 2 consecutive uniform subintervals of type $\textbf{Q, Q}$ let the nodes and weights be defined by the polynomials $Q_n (\textbf{l}, x)$ and $Q_m (Rec_n (\textbf{l}), x)$. The so given quadrature rule is exact for splines with a support of 1 or 2 intervals, when the spline degree $d$ is $\le \min (2n + c, 2m + c)$ 

For 2 consecutive uniform subintervals of type $\textbf{Q, M}$ the nodes and weights are defined by the polynomials $Q_n (\textbf{l}, x)$ and $M_m (Rec_n (\textbf{l}), \textbf{r}, x)$. The so given quadrature rule is exact for splines with a support of 1 or 2 intervals, when the spline degree $d$ is $\le \min (2n + c, 2m - 1)$. The vector $\textbf{r}$ is free, so we can add another subinterval of type $\textbf{Q}$ on the right side and determine \textbf{r} so that $\textbf{Q, M, Q}$ is also exact for the second spline at the right with a support of 2 intervals.

For the first subinterval of type $\textbf{Q}$ or $\textbf{M}$ at the left boundary of the interval
the (left) parameter has to be the zero vector, i.e. we have to use $Q_n (\textbf{0}, x)$ or $M_n (\textbf{0}, \textbf{r}, x)$.  
\end{theorem}

\begin{proof}
I repeat here (\ref{DefectSplineQ}) for the defect
\begin{equation} D (g) = \sum_{i=0}^{c}  (-1)^i l_i \, g^{(i)} (-1) \nonumber \end{equation}
under the condition (\ref{DefectCondQ}) that $g (x)$ is the right side of a spline i.e. $g^{(k)} (+1) = 0$ for $0 \le k \le c$

The condition (\ref{DefectCondQ}) is fulfilled for the \textit{right} side $g_r (x)$ of a spline with a support of $2$ intervals, we get for the defect of this side:
\begin{equation} D (g_r) = \sum_{i=0}^{c}  (-1)^i l_i \, g_r^{(i)} (-1) \nonumber \end{equation}

Now we want to apply this (\ref{DefectSplineQ}) to get the defect of the \textit{left} side $g_l (x)$ of a spline with support of $2$ intervals and to determine the still unknown recursion map $Rec_n$ for which the sum of the left and right defects is $0$. At first we reflect the $Q_n (\textbf{l}, x)$ and $g_l (x)$ at $0$. As a result the nodes and weights of the corresponding quadrature rule are reflected too. This does not change the defect. The parameter $\textbf{l}_R$ of this reflected $Q \, ()$ is given by (\ref{ReflMapR}), the reflected $g_l (x)$ fulfills the condition (\ref{DefectCondQ}), the reflection multiplies the odd derivations $g_l^{(i)} (+1)$ by $-1$. So the defect for this side is:
\begin{equation} D (g_l) = \sum_{i=0}^{c} l_{R,i} \, g_l^{(i)} (+1) \nonumber \end{equation}
where $l_{R,i}$ is the $i^{th}$ component of the vector $\textbf{l}_R$

Because the 2 interval spline $g_l (x), g_r (x)$ is of continuity class $c$:
\begin{equation}  g_l^{(i)} (+1) = g_r^{(i)} (-1) \label{CondSplineCont} \end{equation}
The sum of the defects of both sides of the spline is with the condition above:
$$ D (g_l) + D (g_r) = 0 = \sum_{i=0}^{c} \left( l_{R,i} + (-1)^i l_{Rec,i} \right) g_l^{(i)} (-1) $$
where $l_{Rec,i}$ is the $i^{th}$ component of the vector $Rec_n (\textbf{l})$ belonging the next subinterval. Because this equation has to be valid for arbitrary $g_l^{(i)} (-1)$ we have for all $0 \le i \le c$: 
\begin{equation} 0 =  l_{R,i} + (-1)^i l_{Rec,i} \qquad l_{Rec,i} = (-1)^{i+1} l_{R,i}
\nonumber \end{equation}
which proofs the theorem. \\

For the first subinterval the defect of only the right side of a spline with support of 2 subintervals
has to be zero because on the boundary there are no imposed conditions on continuity. So all coefficients
$l_i$ in  (\ref{DefectSplineQ}) or (\ref{DefectSplineM}) have to be $0$. \\

The proof for 2 consecutive uniform subintervals of type $\textbf{Q, M}$ works in the same manner.
The defect (\ref{DefectSplineM}) for the right side of a spline with a support of 2 subintervals
(with the property $g^{(i)} (+1) = 0$) of a subinterval of type $\textbf{M}$ is equal to the defect (\ref{DefectSplineQ}) of a subinterval of type $\textbf{Q}$.
\end{proof}

\section{ Non-uniform subintervals and stretching \label{SectStretch} }

To treat non-uniform subintervals an extended recursion map allowing stretched intervals is given.

\begin{definition}
\label{DefStrech}
The stretching factor $\lambda$ for a subinterval $s$ is the ratio of the subsequent interval lengths $L_i$: $\lambda = L_{s+1} / L_s$. The monomial stretching map $S_\lambda$ like (S)tretching is defined as:
\begin{align}
S_\lambda : \enspace  & \mathbb{D} & \rightarrow & \enspace \mathbb{D} \nonumber \\
&           ( l_0, \dots , l_{c} ) & \mapsto & \enspace ( l_0 \, / \, \lambda^{1}, \dots , l_{c} \, / \,  \lambda^{c+1} ) \label{StretchMap} \\
& l_i   & \mapsto & \enspace l_i \, / \, \lambda^{i+1}  \nonumber
\end{align}
With this stretching map and the recursion map (\ref{RecMap}) we define the recursion map with
stretching as:
\begin{equation} RecS_{n,\lambda}  = S_\lambda \circ Rec_n = S_\lambda \circ C \circ ( f_n^{-1} \circ r \circ f_n ) \label{RecSMap} \end{equation}
This composition of the maps in (\ref{RecSMap}) can be visualized by this diagram:
\begin{equation}
\underbrace
{ \begin{CD}
  \mathbb{J}   @>r>>    \mathbb{J}             @.                         @.          \\
  @A{f_n}AA             @VV{f_n^{-1}}V         @.                         @.          \\
  \mathbb{D}   @.       \mathbb{D}    @>>C>    \mathbb{D}  @>>S_\lambda>  \mathbb{D}  \nonumber
  \end{CD}
}_{\quad RecS_{n,\lambda} \enspace : \enspace \mathbb{D} \enspace \longrightarrow \enspace \mathbb{D}}
\end{equation}
This recursion map can also be given in the $\mathbb{J}$ domain now named $recs_{n,\lambda}$:
\begin{equation} recs_{n,\lambda}  = s_{n,\lambda} \circ rec_n = ( f_n \circ S_{\lambda} \circ C \circ f_n^{-1} ) \circ r \nonumber \end{equation}
\end{definition}

\begin{theorem}  \label{TheoremRecMapStretch}
For 2 consecutive non-uniform subintervals of type $\textbf{Q, Q}$ let the nodes and weights be defined by the polynomials $Q_n (\textbf{l}, x)$ and $Q_m (RecS_{n,\lambda} (\textbf{l}), x)$. The so given quadrature rule is exact for splines with a support of 1 or 2 intervals, when the spline degree $d$ is $\le \min (2n + c, 2m + c)$

For 2 consecutive non-uniform subintervals of type $\textbf{Q, M}$ the nodes and weights are defined by the polynomials $Q_n (\textbf{l}, x)$ and $M_m (RecS_{n,\lambda} (\textbf{l}), \textbf{r}, x)$. The so given quadrature rule is exact for splines with a support of 1 or 2 intervals, when the spline degree $d$ is $\le \min (2n + c, 2m - 1)$.
\end{theorem}

\begin{proof}
The proof of the theorem above with the extended recursion map $RecS_{n,\lambda}$ now works similar to the proof in the previous section:

There are 2 differences, the first: the right side of a spline with a support of $2$ intervals is now the stretched spline $g_\lambda (x)$. $g_r (x)$ is now $g_\lambda (x)$ scaled in x direction to the default interval length of $2$.

The second difference: because the $g_l (x)$ and now the stretched $g_\lambda (x)$ are continuous at the common interval ends condition (\ref{CondSplineCont}) is different and a factor $\lambda^{i}$ appears:
\begin{equation}  \lambda^{i} \, g_l^{(i)} (+1) = g_r^{(i)} (-1) \nonumber \end{equation}
The sum of the defects of both sides of the spline is because the defect of a stretched interval acquires the factor $\lambda$:
$$ D (g_l) + \lambda \, D (g_r) = 0 = \sum_{i=0}^{c} \left( l_{R,i} + \lambda^{i+1} \, (-1)^i l_{RecS,i} \right) g_l^{(i)} (-1) $$
\begin{equation} 0 =  l_{R,i} + \lambda^{i+1} \, (-1)^i l_{RecS,i} \qquad l_{RecS,i} = (-1)^{i+1} l_{R,i} / \lambda^{i+1} = l_{Rec,i} / \lambda^{i+1} \nonumber \end{equation}
\end{proof}

\section{The fields for the nodes and weights \label{SectFields}}

Let $\mathbb{Q}_K$ be the field obtained by adjoining all knot positions to $\mathbb{Q}$ and $\mathbb{Q}_S$ be the field obtained by adjoining all stretching factors (see definition \ref{DefStrech}) to $\mathbb{Q}$. In the case of 1-parameter suboptimal quadrature rules adjoin also the free parameter $\omega$ to $\mathbb{Q}_S$. \\

For the quadrature rules with continuity class $c = 0$, even degree and $c = 1$, odd degree treated in this article we get the following fields for the nodes and weights:
\begin{equation} x_i \in \mathbb{Q}_K + \mathbb{Q}_K \, \mathbb{Q}_S (r_i) \qquad
                 w_i \in \mathbb{Q}_K \, \mathbb{Q}_S (r_i)
\label{Fields1} \end{equation}
$\mathbb{Q}_S (r_i)$ is the field obtained by adjoining the root $r_i$ to $\mathbb{Q}_S$. $r_i$ is the root of a Q or M polynomial belonging to the index $i$. The coefficients of this polynomial are in $\mathbb{Q}_S$. \\

For the remaining cases $c = 0$, odd degree and $c = 1$ even degree the so call "1/2-rules" are needed. These will be treated in an upcoming article. Here I present just the fields for these cases. For $c = 0$, odd degree these are the same fields as above in (\ref{Fields1}). \\

For $c = 1$ even degree i.e. "1/2-rules" the situation is different. A square root has to be adjoined for each 2 subsequent subintervals with different numbers of nodes. \\

Define the following tower of field extensions by square roots, the field $K_k$ is an extension of $\mathbb{Q}_S$ with degree $2^k$. $K (\sqrt{K})$ means the extension of $K$ with the square root of an element in $K$:
\begin{equation}
\mathbb{Q}_S = K_0 \subset K_0 (\sqrt{K_0}) = K_1 \subset \dots \subset K_{k-1} (\sqrt{K_{k-1}}) = K_k 
\label{Tower} \end{equation}
Let $l$ be half of the number of subintervals left of the "subinterval in the middle" i.e. the number of 2-tuples of subsequent subintervals with different numbers of nodes. Let $r$ be half of the number of subintervals right of \dots. Then we have the following fields:
\begin{align} \begin{split}
 x_i & \in \mathbb{Q}_K + \mathbb{Q}_K \, K_l (r_i)        \qquad \; \; \; w_i \in \mathbb{Q}_K \, K_l (r_i)
   \qquad \; \; \; \hbox{left of the middle} \\
 x_i & \in \mathbb{Q}_K + \mathbb{Q}_K \, K_l (K_r) (r_i)  \quad  w_i  \in \mathbb{Q}_K \, K_l (K_r) (r_i)
   \quad \hbox{the middle}         \\
 x_i & \in \mathbb{Q}_K + \mathbb{Q}_K \, K_r (r_i)        \qquad \; \; \; w_i \in \mathbb{Q}_K \, K_r (r_i)
   \qquad \; \; \; \hbox{right of the middle}
\label{Fieldsc1_1_2}
\end{split} \end{align}
$K_l (r_i), K_l (K_r) (r_i), K_l (r_i)$ are the fields obtained by adjoining the root $r_i$ to $K_l, K_l (K_r), K_r$. $r_i$ is the root of a Q or M polynomial belonging to the index $i$. The coefficients of this polynomial are in $K_l, K_l (K_r), K_r$. The field $K_l (K_r) = K_r (K_l)$ has degree $2^{\lfloor s/2 \rfloor}$ over $\mathbb{Q}_S$, $s$ is the number of subintervals.

\section{Conclusions}

With this approach, many of the questions that arose in the concluding remarks of section 5 in  \citet{Ait2020} could be answered. Even if the conjecture in section \ref{SectReflMap} is not proven for general continuity class, with the help of a computer algebra system it should be possible to determine the recursion map symbolically for the continuity classes $c = 2, 3$. This could answer questions about the existence of optimal uniform quadrature rules on the real line in these cases. I find two particularly interesting questions: For uniform optimal rules in the closed interval and $c = 3$, is the convergence quadratic for all $n$ as it is the case for $c = 1$? What about the fields (see previous section) for the 
$c = 2, 3$ "1/2-rules"? \newline

\noindent \textbf{\large Appendices}

\appendix

\section{The current state of the conjecture 4.1 \label{ConjState}}

\begin{itemize}
   \item for continuity class $c \ge 3$: \\
         Formulae (\ref{QasSumJ}) and (\ref{ReflMapForm}) are conjectured.
   \item for continuity class $c = 2$: \\
         Formulae (\ref{QasSumJ}) and (\ref{ReflMapForm}) were proven by a computer algebra system = CAS
         for \textit{certain} values of the degree $n$ up to $n = 8$. But the reflection map was not
         calculated for a general $n$ as \textit{variable} as e.g. in appendix \ref{c1Formulae}.
   \item for continuity class $c = 1$: \\
         Formulae (\ref{QasSumJ}) and (\ref{ReflMapForm}) are proven by a CAS.
   \item for continuity class $c = 0$: \\
         In \citet{Chihara1985} a proof of (\ref{QasSumJ}) can be found, because the
         map is linear fractional in this case formula (\ref{ReflMapForm}) follows. Of course this can
         also be proven by a CAS.
\end{itemize}

\section{A connection between the reflection map $R_n$ and Cremona groups \label{ConCremona}}

The map $R_n$ is birational and homogenized it is a group element in the Cremona group $Bir \left( \mathbb{P}_\mathbb{C}^{c+1} \right)$, the group of birational maps in the projective space $\mathbb{CP}^{c+1}$, see \citet{Des2012}. $R_n$ depends on the parameter $n \in \mathbb{N}_+$. Because $R_n$ is rational in $n$ we can take also $n \in \mathbb{R} \setminus \, \{ \, \dots \, \}$. So we get a 1-parameter involution in the Cremona group. The connection map $c_n$ also represents a 1-parameter involution in the Cremona group.  

Maybe this kind of involutions of the form (\ref{ReflMapForm}) or the birational maps $f_n$ (\ref{MapTrans}) are known in the theory of Cremona groups. Then we could find something about the conjecture \ref{ConjMain}.

If this kind of involution is not known, maybe this 1-parameter $R_a$ in $a \in \mathbb{R}$ is interesting in the theory of Cremona groups.

The explicit formulae for the involution $R_n$ in $Bir \left( \mathbb{P}_\mathbb{C}^{1} \right)$ you can find in appendix \ref{c0Formulae} and  for $R_n$ in $Bir \left( \mathbb{P}_\mathbb{C}^{2} \right)$ in  appendix \ref{c1Formulae}.

\section{The formulae for continuity class $c=0$ \label{c0Formulae}}

In this case the vector space $\mathbb{D}$ and the projective space $\mathbb{J}$ are 1-dimensional.
Let $\textbf{l} = (l_0), \textbf{r} = (r_0) \in \mathbb{D}$ and $\textbf{j} = (j_0 : j_1) \in \mathbb{J}$.

Here we express the $Q_n (\textbf{l}, x)$ still as a sum over orthogonal Gegenbauer polynomials and not as in (\ref{QasSumJ}) as sum over the orthogonal symmetric Jacobi polynomials. But these polynomials differ just by a factor depending only on $n$: $P_n^{(c+1,c+1)}(x) = \alpha (n) \, C_n^{(3/2+c)}(x)$.

The Gegenbauer polynomials are used, because when these formulae were created, the connection of quadrature rules for splines to the Jacobi polynomials was unknown to me. Gegenbauer polynomials with negative index are defined as $0$.
 
The one-sided semi-classical orthogonal Jacobi polynomial $Q_n (\textbf{l}, x)$:
$$ F (n) = 1 + n (n + 1) / 2 \, l_0 $$
\begin{equation}
Q_n (\textbf{l}, x) = \frac{F (n) \, C_{n}^{(3/2)} (x) + F (n + 1) \, C_{n-1}^{(3/2)} (x)} {n+1}
\label{c0Q} \end{equation}
Via formula (\ref{QasSumJ}) this defines the maps $f_n$ and $f_n^{-1}$, see (\ref{MapTrans}) and (\ref{MapTransInv}).

\begin{equation}
w_i = \frac{ 2 (2 n + 1) F^2 (n) }{ n (n + 1) Q_n' (\textbf{l}, x_i) \, Q_{n-1} (\textbf{l}, x_i) (1 - x_i) }
\label{c0Qw} \end{equation}

The (linear fractional in $\textbf{l}$) recursion map $Rec_n \, (\textbf{l})$ without stretching:
$$ \Gamma (n) = (n + 1)^2 (1 + n (n + 2) / 2 \, l_0) $$
\begin{equation} l_0 \mapsto - \, l_0 + \frac{ 2 \, F (n) F (n + 1) } { \Gamma (n) }
                     \enspace = \frac{ 2 + (n + 1)^2 \, l_0 } { \Gamma (n) }
                     \enspace \hbox{linear fractional in} \enspace l_0 
\label{c0Rec} \end{equation}
This map and the reflection map $R_n = C \circ Rec_n$, for $C$ see (\ref{ConnMap}), are not only defined for $n \in \mathbb{N}_+$ but for $n \in \mathbb{R} \, \setminus \, \{  \, -1 \, \}$  \newline

The two-sided semi-classical orthogonal Jacobi polynomial $M_n (\textbf{l}, \textbf{r}, x)$:

$$ H (n) = 1 + n^2 / 2 (l_0 + r_0 + (n - 1) (n + 1) / 2 \, l_0 r_0) $$
\begin{equation}
M_n (\textbf{l}, \textbf{r}, x) = \frac{H (n) \, C_{n}^{(3/2)} (x) - H (n + 1) \, C_{n-2}^{(3/2)} (x)} {2 n+1} + \frac{(l_0 - r_0) \, C_{n-1}^{(3/2)} (x)} {2}
\label{c0M} \end{equation}
\begin{equation}
M_{n,\omega} (\textbf{l}, \textbf{r}, x) = M_n (\textbf{l}, \textbf{r}, x) + \omega M_{n-1} (\textbf{l}, \textbf{r}, x)
\label{c0Momega} \end{equation}
\begin{equation}
w_i = \frac{ 2 \, H^2 (n) }{ n \, M_{n,\omega}' (\textbf{l}, \textbf{r}, x_i) \, M_{n-1} (\textbf{l}, \textbf{r}, x_i) }
\label{c0Mw} \end{equation} \newline

In \cite{Ait2020}, section 3.2 for even $n$ a positive, optimal and 1-periodic quadrature rule for the real line is presented. It is left the reader as exercise to derive this quadrature rule with the formulae given here. In a first step get the fixed point $\textbf{l}_F$ of the recursion map and then calculate the nodes as roots of $Q_n (\textbf{l}_F, x)$  and the weights with the formula above. \\

The connection map $c_n \, (\textbf{j})$, the transformed $C$, see (\ref{ConnMap}) is a linear, degree 1 involution in the projective space $\mathbb{J}$, i.e. an element of $Bir \left( \mathbb{P}_\mathbb{C} \right)$:
\begin{align} \begin{split}
 & j_0 \mapsto (n + 1) \, j_0 -       n \, j_1 \\
 & j_1 \mapsto (n + 2) \, j_0 - (n + 1) \, j_1 \nonumber
\end{split} \end{align}

\section{The formulae for continuity class $c=1$ \label{c1Formulae}}

In this case the vector space $\mathbb{D}$ and the projective space $\mathbb{J}$ are 2-dimensional.
Let $\textbf{l} = (l_0, l_1), \textbf{r} = (r_0, r_1) \in \mathbb{D}$ and $\textbf{j} = (j_0 : j_1 : j_2) \in \mathbb{J}$.

The one-sided semi-classical orthogonal Jacobi polynomial $Q_n (\textbf{l}, x)$:
$$ E (n) = 1 + (n + 1) (n + 2) (l_0 + 3 n (n + 3) (2 - (n - 1) (n + 1) (n + 2) (n + 4) \, l_1 )  \, l_1) $$
$$ F (n) = 1 + n (n + 2) (l_0 + 6 (n^2 + 2 n - 1) \, l_1 - 3 (n - 1) n (n + 1)^2 (n + 2) (n + 3) \, l_1^2) $$
$$ Q_n (\textbf{l}, x) =
   \frac{6 \, F (n)   \, C_{n}^{(5/2)} (x)} {(n + 2) (2 n + 3)}
 + \frac{6 \, E (n)   \, C_{n-1}^{(5/2)} (x)} {(n + 1) (n + 2)}
 + \frac{6 \, F (n+1) \, C_{n-2}^{(5/2)} (x)} {(n + 1) (2 n + 3)} $$
Via formula (\ref{QasSumJ}) this defines the maps $f_n$ and $f_n^{-1}$, see (\ref{MapTrans}) and (\ref{MapTransInv}).

$$ w_i = \frac{ 8 (n + 1) F^2 (n) }{ n (n + 2) Q_n' (\textbf{l}, x_i) \, Q_{n-1} (\textbf{l}, x_i) (1 - x_i)^2 } $$

The recursion map $Rec_n \, (\textbf{l})$ without stretching:
\begin{align} \begin{split}
G_0 (n) = & 4 \, (2 n^2 + 6 n + 3) + n (n + 3) \, ( (11 n^2 + 33 n + 16) \, l_0 \\
 & \quad + 24 \, (2 n^4 + 12 n^3 + 17 n^2 - 3 n - 4) \, l_1 \\
 & \quad - 3 n (n + 1) (n + 2) (n + 3) (4 \, (n + 1) (n + 2) (2 n^2 + 6 n - 5) \, l_1^2 \\
 & \qquad   + 3 \, (n^2 - 1) n (n + 2) (n + 3) (n + 4) \, l_0 l_1^2 \\
 & \qquad   - 6 \, (n^2 + 3 n - 2) \, l_0 l_1  - l_0^2 ) ) \nonumber
\end{split} \end{align}
\begin{align} \begin{split}
G_1 (n) & = 1 - 3 n (n + 1) (n + 2) (n + 3) \, l_1 \nonumber 
\end{split} \end{align}
\begin{align} \begin{split}
\Gamma (n) =  (n + 1) (n + 2) (& 1 + n (n + 3) \, l_0 + 6 n (n + 3) (n^2 + 3 n - 1) \, l_1 \\
 & - 3 (n^2 - 1) n^2 (n + 2) (n + 3)^2 (n + 4) \, l_1^2) \nonumber
\end{split} \end{align}
\begin{align} \begin{split}
 & l_0 \mapsto   - \, l_0 + \frac{ E (n) \, G_0 (n) } { 3 \, \Gamma (n)^2 } \\
 & l_1 \mapsto  \quad l_1 + \frac{ E (n) \, G_1 (n) } { 3 (n + 1) (n + 2) \, \Gamma (n) } \nonumber
\end{split} \end{align}
This map and the reflection map $R_n = C \circ Rec_n$, for $C$ see (\ref{ConnMap}), are not only defined for $n \in \mathbb{N}_+$ but for $n \in \mathbb{R} \, \setminus \, \{  \, -1, -2 \, \}$.
The homogenized reflection map $R_n$, the transformed $r$, see (\ref{ReflMap}) is a quartic, degree 4 involution in the projective space $\mathbb{RP}^2$, i.e. an element of $Bir \left( \mathbb{P}_\mathbb{C}^2 \right)$: \\

The two-sided semi-classical orthogonal Jacobi polynomial $M_n (\textbf{l}, \textbf{r}, x)$:

$$ H_a (n, \textbf{d}) = 1 + (n - 1) n (d_0 + (n - 2) (n + 1) (6 - 3 (n - 3) (n - 1) n (n + 2) \, d_1) \, d_1) $$
\begin{align} \begin{split}
H (n) = \, & (  H_a (n, \textbf{l}) H_a (n + 1, \textbf{r})
              + H_a (n, \textbf{r}) H_a (n + 1, \textbf{l})) / 2 \\
           & - 36 \, (n^2 - 1) n^2 (l_1 - r_1)^2
\nonumber \end{split} \end{align}
$$ J_0 (n, \textbf{d}) = 1 + (n^2 + n + 3) \, d_0 + 6 \, (n^4 + 2 n^3 + n^2 + 6) \, d_1
                         - 3 \, (n^2 - 9) (n^2 - 4) (n^2 - 1) n (n + 4) \, d_1^2  $$
$$ J_1 (n, \textbf{d}) = 1 + n (n + 1) (d_0 + 3 \, (n - 1) (n + 2) (2 - (n - 2) n (n + 1) (n + 3) \, d_1) \, d_1) $$
\begin{align} \begin{split}
J (n) = \, & (J_0 (n, \textbf{l}) J_1 (n, \textbf{r}) + J_0 (n, \textbf{r}) J_1 (n, \textbf{l})) / 2 \\
           & + 108 \, (n^2 - 1) n (n + 2) (l_1 - r_1)^2
\nonumber \end{split} \end{align}
\begin{align} \begin{split}
D (n)   & = (l_0 - r_0) \, (2 - 3 \, (n^2 - 1) n (n + 2) \, (l_1 + r_1)) \\
D_1 (n) & = D (n) \, (2 - 3 \,  (n - 2) (n^2 - 1) n \, (l_1 + r_1))       \\
D_3 (n) & = D (n) \, (2 - 3 \, n (n + 1) (n + 2) (n + 3) \, (l_1 + r_1))
\nonumber \end{split} \end{align}
\begin{align} \begin{split}
F_{13} (n) = 3 \, (l_1 - r_1) \, n^2 (& \quad 16 + 4 \, (n^2 - 1) (l_0 + r_0 - 4 \, (n^2 - 1) (l_1 + r_1)) \\
                                      & - 3 \, (n^2 - 4) (n^2 - 1)^2 \, (3 \, l_0 r_1 + 3 \, l_1 r_0 +
                                                                             l_0 l_1 + r_0 r_1 \\
                                      & \qquad \qquad \qquad \qquad \qquad + 16 \, (n^2 - 6) \, l_1 r_1 ) )
\nonumber \end{split} \end{align}
\begin{align} \begin{split}
M_n (\textbf{l}, \textbf{r}, x)  = & \frac{3 \, H (n)     \, C_{n}^{(5/2)}   (x)} {(2 n + 1) (2 n + 3)}
                                   - \frac{6 \, J (n)     \, C_{n-2}^{(5/2)} (x)} {(2 n - 1) (2 n + 3)}
                                   + \frac{3 \, H (n + 1) \, C_{n-4}^{(5/2)} (x)} {(2 n - 1) (2 n + 1)} \\
                      & + \frac{3}{4} \, \frac{  (D_1 (n) + F_{13} (n))    \, C_{n-1}^{(5/2)} (x)
                                               - (D_3 (n) + F_{13} (n+1)) \, C_{n-3}^{(5/2)} (x)}{2 n + 1} 
\nonumber \end{split} \end{align}
\begin{equation}
w_i = \frac{ 2 \, H^2 (n) }{ n \, M_n' (\textbf{l}, \textbf{r}, x_i) \, M_{n-1} (\textbf{l}, \textbf{r}, x_i) }
\nonumber \end{equation} \\

The connection map $c_n \, (\textbf{j})$, the transformed $C_n$, see (\ref{ConnMap}) is a quadratic, degree 2 involution in the projective space $\mathbb{J}$, i.e. an element of $Bir \left( \mathbb{P}_\mathbb{C}^2 \right)$:
\begin{align} \begin{split}
  B      (n) = \enspace  & - (n + 3) \, j_0 + n         \, j_2 \\
  \Delta (n) = \enspace  & \quad (n + 3) (2 n^4 + 18 n^3 + 49 n^2 + 48 n + 18) \, j_0^2 + n^2 (n + 3) (2 n^2 + 6 n + 1) \, j_1^2  \\
                         & + n^2 (n - 1) (2 n^2 + 2 n - 3) \, j_2^2 - 2 n (n + 2) (n + 3) (2 n^2 + 8 n + 3) \, j_0 \, j_1 \\
                         & + 2 n (2 n^4 + 12 n^3 + 25 n^2 + 15 n - 9) \, j_0 \, j_2 - 2 n^2 (n + 1) (2 n^2 + 4 n - 3) \, j_1 \, j_2 \nonumber
\end{split} \end{align}
\begin{align} \begin{split}
 & j_0 \mapsto n \, \Delta \\
 & j_1 \mapsto (2 n + 3) (\Delta + 6 B \, ((2 n + 3) \, j_0 - n \, j_1) \\
 & j_2 \mapsto (n + 3) \, \Delta - 6 (2 n + 3) \, B^2 \nonumber
\end{split} \end{align}
For the classification of quadratic involutions in the Cremona group $Bir \left( \mathbb{P}_\mathbb{C}^2 \right)$ see \citet{Des2012}, section 4. The connection map above is the unique (up to linear transformations) quadratic involution with $detjac$ the union of three lines in general position, see Theorem 4.2.2 in \citet{Des2012}. This is the Veronese map $\sigma : (x : y : z) \dashrightarrow (y z : z x : xy)$.  \\

\noindent Exercise for the reader:

Find the two positive, optimal and 1-periodic quadrature rules given in \citet{Ait2020} in sections 4.1 and 4.2 for odd $n$ as fixed points of this recursion map. Show that among the 4 finite fixed points the rational fixed point (hint: each finite fixed point fulfills $E (n) = 0$ or $G_1 (n) = 0$) with
positive $l_1$ component is a quadratic attractor, this means quadratic convergence for all uniform 
$C^1$ rules in a closed interval. Show that the 2 rational fixed points generate the same quadrature
rule as in \citet{Ait2020}, section 4.1. The quadratic convergence for uniform $C^1$ quintic rules in
a closed interval is already proven in \citet{BAC2017}.

\section{Some examples for suboptimal $C^0$ quadrature rules \label{c0Example}}

Some facts about $C^0$ quadrature rules for splines with \textit{even} degree $2n$ in the closed
interval $[a, b]$: \\
There exist only \textit{suboptimal} rules depending on a free parameter. The node distribution:
all subintervals except one have $n$ nodes, the remaining one has $n+1$ nodes.
Here I call the subinterval with $n+1$ nodes "subinterval in the middle" though it can be
one at the boundary of $[a, b]$ too. In the contrary to cases with $c \ge 1$ we get here
real (not only algebraic) quadrature rules for all positions of the "subinterval in the middle"
even for the case of arbitrary non-uniform subintervals. \\

In the first example we take as interval $[a, b]$ the interval $[0, 4]$. This interval is devided in the
4 uniform subintervals 1 .. 4 of length 1. Subinterval 3 is the one in the middle. Subintervals 1, 2, 4
have 2 nodes, subinterval 3 has 3 nodes, so the suboptimal quadrature rule has degree 4.  \\

Because here the quadrature rules are suboptimal, we have a free parameter say $\omega$ that affects
only the "subinterval in the middle". Instead of taking the polynomial $M_{n+1} ()$ to get the nodes of this interval, we can use $M_{n+1} () + \omega M_n ()$. So we have the freedom to position one node e.g at the boundary of the subinterval.

\begin{remark}
This construction is already known for the case of the classical polynomial quadrature with Legendre
orthogonal polynomials $P_{n} (x)$. For even degree $2n$ the quadrature rule is suboptimal and
we have the freedom to take $P_{n} (x) + \omega P_{n-1} (x)$ as polynomial defining the nodes.
\end{remark}

Here the steps to construct the rule with the formulae in section \ref{c0Formulae}:
\begin{enumerate}[label=(\roman*)]
   \item Set the $\textbf{0}$-vectors of Dirac $\delta$ coefficients for the subintervals at the
         boundary $\textbf{l}_1 = \textbf{0}$ and $\textbf{r}_4 = \textbf{0}$.
         Calculate the remaining vectors $\textbf{l}_2 = Rec_2 (\textbf{l}_1),
         \textbf{l}_3 = Rec_2 (\textbf{l}_2), \textbf{r}_3 = Rec_2 (\textbf{r}_4)$ with formula
         (\ref{c0Rec}) for the recursion map. Step from the left and right to the subinterval in
         the "middle" which so gets 2 vectors, a left and right one. 
   \item Calculate the unscaled nodes for the subintervals 1, 2, 4 as roots of the polynomial (\ref{c0Q})
         and the unscaled weights with formula (\ref{c0Qw}).  
   \item For the subinterval 3 in the "middle" calculate $\omega$ to position one node, then get the
         nodes as roots of the polynomial (\ref{c0Momega}) and the weights with formula (\ref{c0Mw}).
         We have to choose $\omega = - M_{n+1} (\textbf{l}_3, \textbf{r}_3, -1) / M_n (\textbf{l}_3,
         \textbf{r}_3, -1)$ to get a node at $-1$.  
   \item Scale the nodes and weights to the destination subintervals: \\
         Scaling from [-1, +1] to $[c, d]$ maps: $x_i \mapsto (x_i (d - c) + (c + d)) / 2$, \\
         $w_i \mapsto w_i (d - c) / 2$ 
\end{enumerate}

\begin{remark}
To see that the degree of the so constructed quadrature rule is $2n$ look at the condition in
theorem (\ref{TheoremRecMap}) for the subintervals of type $\textbf{Q}$. A function in these subintervals
is integrated exactly if the degree is $\le 2 n + c$. See the condition in the theorem for the subinterval
of type $\textbf{M}$ with one more node. A function in this subinterval is integrated exactly if the
degree is $\le 2 (n + 1) - 1$. Because we are using $M_{n+1,\omega}$ instead of $M_{n+1}$ we have to
lower this degree by 1. So the degree of the whole rule is $2n$.  
\end{remark}

See table \ref{TabDeg4} with nodes and weights as algebraic numbers for this quadrature rule with degree 4.
Only square roots are needed, because the rule is suboptimal and so node 5 can be placed at a rational $x_i$. \\

\begin{table}[!h]
  \begin{center}
    \caption{A suboptimal, uniform quadrature rule with degree 4 for continuity class $c = 0$,
             4~subintervals: $[0, 1] \enspace [1, 2] \enspace [2, 3]^\textbf{m} \enspace [3, 4]$}
    \label{TabDeg4}
    \begin{tabular}{ c || c | c | c || c | c | }
      \multicolumn{6}{c}{}  \\ 
      \textbf{i} & \textbf{s} & \textbf{coeff.} $\mathbf{\delta}$ & \textbf{type}
                 & $\mathbf{x_i}$ & $\mathbf{w_i}$ \\
      \hline
          1     & 1   & $\textbf{l}_1 = (0)$  & $Q_2 (\textbf{l}_1, x)$
                                                          & $\frac{2}{5} - \frac{\sqrt{6}}{10}$
                                                               & $\frac{4}{9} - \frac{\sqrt{6}}{36}$  \\
          2     &     &                       &           & $\frac{2}{5} + \frac{\sqrt{6}}{10}$
                                                               & $\frac{4}{9} + \frac{\sqrt{6}}{36}$  \\
      \hline
          3     & 2   & $\textbf{l}_2 = (2/9)$ & $Q_2 (\textbf{l}_2, x)$
                                                     & $\frac{34}{25} - \frac{\sqrt{174}}{50}$
                                                       & $\frac{76}{153} - \frac{21 \sqrt{174}}{5916}$  \\
          4     &     &                       &      & $\frac{34}{25} + \frac{\sqrt{174}}{50}$
                                                       & $\frac{76}{153} + \frac{21 \sqrt{174}}{5916}$  \\
      \hline
          5     & 3   & $\textbf{l}_3 = (4/17)$  
                                              & $M_3 (\textbf{l}_3, \textbf{r}_3, x)$
                                                        & $ \small 2$
                                                           & $\frac{4}{17}$                               \\
          6     &     & $\textbf{r}_3 = (2/9)$  & $ + \omega M_2 (\textbf{l}_3, \textbf{r}_3, x)$
                                                        & $\frac{66}{25} - \frac{\sqrt{174}}{50}$
                                                           & $\frac{76}{153} + \frac{7\sqrt{174}}{1972}$  \\
          7     &     & $\omega = 7/5$       &          & $\frac{66}{25} + \frac{\sqrt{174}}{50}$
                                                           & $\frac{76}{153} - \frac{7\sqrt{174}}{1972}$  \\
      \hline
          8     & 4   & $\textbf{r}_4 = (0)$  & $Q_2 (\textbf{r}_4, - x)$ refl. 
                                                          & $\frac{18}{5} - \frac{\sqrt{6}}{10}$
                                                               & $\frac{4}{9} + \frac{\sqrt{6}}{36}$  \\
          9     &     &                       &           & $\frac{18}{5} + \frac{\sqrt{6}}{10}$
                                                               & $\frac{4}{9} - \frac{\sqrt{6}}{36}$  \\
      \hline
    \end{tabular}
  \end{center}
\end{table}

See table \ref{TabDeg6} with nodes and weights as floating point numbers for a quadrature rule with
degree 6.
 
\begin{table}[!h]
  \begin{center}
    \caption{A suboptimal, uniform quadrature rule with degree 6 for continuity class $c = 0$,
             4~subintervals: $[0, 1]^\textbf{m} \enspace [1, 2] \enspace [2, 3] \enspace [3, 4]$}
    \label{TabDeg6}
    \begin{tabular}{ c || c | c | c || c | c | }
      \multicolumn{6}{c}{}  \\ 
      \textbf{i} & \textbf{s} & \textbf{coeff.} $\mathbf{\delta}$ & \textbf{type}
                 & $\mathbf{x_i}$ & $\mathbf{w_i}$ \\
      \hline
          1     & 1   & $\textbf{l}_1 = (0)$  
                                              & $M_4 (\textbf{l}_1, \textbf{r}_1, x)$
                                                          & $0$
                                                               & $0.0645497136$  \\
          2     &     & $\textbf{r}_1 = (63/488)$  & $ + \omega M_3 (\textbf{l}_1, \textbf{r}_1, x)$
                                                          & $0.2193254677$
                                                               & $0.3397035713$  \\
          3     &     & $\omega = 559/433$ &              & $0.6102277570$
                                                               & $0.4016942462$  \\
          4     &     &                    &              & $0.9470881476$
                                                               & $0.2586016489$  \\
      \hline
          5     & 2   & $\textbf{r}_2 = (4/31)$  & $Q_3 (\textbf{r}_2, - x)$ refl.
                                                          & $1.2193236472$
                                                               & $0.3397007352$  \\
          6     &     &                       &           & $1.6102225842$
                                                               & $0.4016906147$  \\
          7     &     &                       &           & $1.9470771451$
                                                               & $0.2585755986$  \\
      \hline
          8     & 3   & $\textbf{r}_3 = (1/8)$  & $Q_3 (\textbf{r}_3, - x)$ refl.
                                                          & $2.2192108353$
                                                               & $0.3395249876$  \\
          9     &     &                       &           & $2.6099020423$
                                                               & $0.4014656053$  \\
         10     &     &                       &           & $2.9463973263$
                                                               & $0.2569932780$  \\
      \hline
         11     & 4   & $\textbf{r}_4 = (0)$  & $Q_3 (\textbf{r}_4, - x)$ refl.
                                                          & $3.2123405382$
                                                               & $0.3288443199$  \\
         12     &     &                       &           & $3.5905331355$
                                                               & $0.3881934688$  \\
         13     &     &                       &           & $3.9114120404$
                                                               & $0.2204622111$  \\
      \hline
    \end{tabular}
  \end{center}
\end{table}

Another example now with non-uniform subintervals is table \ref{TabNonUniformDeg4}. Here the
quadrature rule is defined in the interval $[a, b] = [0, 15]$. The non-uniform subintervals are
$[0, 1] \enspace [1, 3] \enspace [3, 7] \enspace [7, 15]$ i.e. a stretching $\lambda$ factor of 2.
The subinterval in the "middle" is here subinterval 4 at the right boundary. In step 1)
when calculating the Dirac $\delta$ coefficients for this non-uniform rule instead of $\textbf{l}_2 = Rec_2 (\textbf{l}_1) \dots$ the recursion map \textit{with stretching} has to be used:
$\textbf{l}_2 = RecS_2 (\textbf{l}_1, \lambda) \dots$.

\begin{remark}
For continuity class $c=0$ and even degree $2n$ of the suboptimal quadrature rule we get for each number of
subintervals $s$ different quadrature rules ($s$~rules depending on the selection of the subinterval in
the "middle"). This rules can be still distinguished by the combinatorial node distribution (because
one subinterval has $n+1$ nodes).
The case of odd degree $2n+1$ and odd $s$, a so called 1/2 optimal rule, is not shown here. In this case
we get $\lceil s \rceil$ different optimal rules (alternating $n+1$ and $n$ nodes) which \textit{can not}
be distinguished by the combinatorial node distribution. This has to be taken in account when using
existing numerical solvers (e.g. as in \citet{BC2016} by homotopy methods).
\end{remark}

\begin{table}[!h]
  \begin{center}
    \caption{A suboptimal, non-uniform quadrature rule with degree 4 for continuity
             class $c = 0$,
             4~subintervals: $[0, 1] \enspace [1, 3] \enspace [3, 7] \enspace [7, 15]^\textbf{m}$,
             stretching factor $\lambda = 2$}
    \label{TabNonUniformDeg4}
    \begin{tabular}{ c || c | c | c || c | c | }
      \multicolumn{6}{c}{}  \\ 
      \textbf{i} & \textbf{s} & \textbf{coeff.} $\mathbf{\delta}$ & \textbf{type}
                 & $\mathbf{x_i}$ & $\mathbf{w_i}$ \\
      \hline
          1     & 1   & $\textbf{l}_1 = (0)$  & $Q_2 (\textbf{l}_1, x)$
                                                          & $\frac{2}{5} - \frac{\sqrt{6}}{10}$
                                                               & $\frac{4}{9} - \frac{\sqrt{6}}{36}$  \\
          2     &     &                       &           & $\frac{2}{5} + \frac{\sqrt{6}}{10}$
                                                               & $\frac{4}{9} + \frac{\sqrt{6}}{36}$  \\
      \hline
          3     & 2   & $\textbf{l}_2 = (1/9)$ &  $Q_2 (\textbf{l}_2, x)$
                                                          & $\frac{7}{4} - \frac{\sqrt{105}}{20}$
                                                               & $\frac{110}{117} - \frac{10 \sqrt{105}}{819}$  \\
          4     &     &                       &           & $\frac{7}{4} + \frac{\sqrt{105}}{20}$
                                                               & $\frac{110}{117} + \frac{10 \sqrt{105}}{819}$  \\
      \hline
          5     & 3   & $\textbf{l}_3 = (3/26)$  & $Q_2 (\textbf{l}_3, x)$ 
                                                          & $\frac{787}{175} - \frac{2 \sqrt{8061}}{175}$
                                                               & $\frac{4189}{2223} - \frac{16522 \sqrt{8061}}{5973201}$  \\
          6     &     &                       &           & $\frac{787}{175} + \frac{2 \sqrt{8061}}{175}$
                                                               & $\frac{4189}{2223} + \frac{16522 \sqrt{8061}}{5973201}$  \\
      \hline
          7     & 4   & $\textbf{l}_4 = (79/684)$  
                                              & $M_3 (\textbf{l}_4, \textbf{r}_4, x)$
                                                          & $ \small 7$
                                                               & $\frac{77}{57}$  \\
          8     &     & $\textbf{r}_4 = (0)$  & $ + \omega M_2 (\textbf{l}_4, \textbf{r}_4, x)$
                                                          & $\frac{59}{5} - \frac{4 \sqrt{6}}{5}$
                                                               & $\frac{32}{9} + \frac{2 \sqrt{6}}{9}$  \\
          9     &     & $\omega = 1$           &          & $\frac{59}{5} + \frac{4 \sqrt{6}}{5}$
                                                               & $\frac{32}{9} - \frac{2 \sqrt{6}}{9}$  \\
      \hline
    \end{tabular}
  \end{center}
\end{table}

\section{Some examples for optimal $C^1$ Gaussian quadrature rules \label{c1Examples}}

Some facts about $C^1$ quadrature rules for splines with \textit{odd} degree $2n + 1$ in the closed
interval $[a, b]$: \\
The node distribution of these Gaussian, optimal rules:
all subintervals except one have $n$ nodes, the remaining one has $n+1$ nodes.
Here I call the subinterval with $n+1$ nodes "subinterval in the middle". In the contrary to cases
with $c = 0$ we get here real (not only algebraic) uniform quadrature rules only for positions of the "subinterval in the middle" now really near the middle or in the middle. \\

See table \ref{TabDeg5c1} for a symmetric rule and table \ref{TabDeg7c1} for an asymmetric rule.
Now in the steps (i) - (iii) to construct the rules (see appendix \ref{c0Example}) the formulae in
appendix \ref{c1Formulae} have to be used. Because these $C^1$ rules are optimal and so do not
have a free parameter, $\omega$ in step (iii) has to be set to $0$.

\begin{table}[!h]
  \begin{center}
    \caption{An optimal, uniform, symmetric quadrature rule with degree 5 for continuity class $c = 1$,
             5~subintervals: $[0, 1] \enspace [1, 2] \enspace [2, 3]^\textbf{m} \enspace [3, 4]$
                             $\enspace [4, 5]$, the symmetric nodes 7 - 11 are not shown}
    \label{TabDeg5c1}
    \begin{tabular}{ c || c | c | c || c | c | }
      \multicolumn{6}{c}{}  \\ 
      \textbf{i} & \textbf{s} & \textbf{coeff.} $\mathbf{\delta}$ & \textbf{type}
                 & $\mathbf{x_i}$ & $\mathbf{w_i}$ \\
      \hline
          1     & 1   & $\textbf{l}_1 = (0, 0)$  & $Q_2 (\textbf{l}_1, x)$
                                                     & $\frac{1}{3} - \frac{\sqrt{10}}{15}$
                                                          & $\frac{85}{216} - \frac{25 \sqrt{10}}{864}$  \\
          2     &     &                       &      & $\frac{1}{3} + \frac{\sqrt{10}}{15}$
                                                          & $\frac{85}{216} + \frac{25 \sqrt{10}}{864}$  \\
      \hline
          3     & 2   & $\textbf{l}_2 = (23/108, 1/432)$ & $Q_2 (\textbf{l}_2, x)$
                                                     & $\frac{465}{371} - \frac{\sqrt{209770}}{1855}$
                                                       & $\frac{972835}{20357784} -
 \frac{53657125 \sqrt{209770}}{569393646624}$  \\
          4     &     &                       &      & $\frac{465}{371} + \frac{\sqrt{209770}}{1855}$
                                                       & $\frac{972835}{20357784} +
 \frac{53657125 \sqrt{209770}}{569393646624}$  \\
      \hline
          5     & 3   & $\textbf{l}_3 = (593446/2544723,$  
                                              & $M_3 (\textbf{l}_3, \textbf{r}_3, x)$
                                                        & $\frac{5}{2} - \frac{\sqrt{11868463}}
{2 \sqrt{11870305}}$
                                                           & $\frac{28180828158605}{60403901541498}$  \\
          6     &     & $ \qquad 23/8289), \textbf{r}_3 = \textbf{l}_3$  & 
                                                        & $\frac{5}{2}$
                                                           & $\frac{18989540}{35605389}$  \\
      \multicolumn{6}{c}{\dots} \\
      \hline
    \end{tabular}
  \end{center}
\end{table}

\begin{table}[!h]
  \begin{center}
    \caption{An optimal, non-uniform, asymmetric quadrature rule with degree 7 for continuity class $c = 1$,
             4~subintervals: $[0, 1] \enspace [1, 3] \enspace [3, 7]^\textbf{m} \enspace [7, 9]$,
                             stretching factor $\lambda = 2$}
    \label{TabDeg7c1}
    \begin{tabular}{ c || c | c | c || c | c | }
      \multicolumn{6}{c}{}  \\ 
      \textbf{i} & \textbf{s} & \textbf{coeff.} $\mathbf{\delta}$ & \textbf{type}
                 & $\mathbf{x_i}$ & $\mathbf{w_i}$ \\
      \hline
          1     & 1   & $\textbf{l}_1 = (0, 0)$  & $Q_3 (\textbf{l}_1, x)$
                                                     & $0.0729940240$  & $0.1828570141$  \\
          2     &     &                       &      & $0.3470037660$  & $0.3429757724$  \\
          3     &     &                       &      & $0.7050022098$  & $0.3441672133$  \\
      \hline
          4     & 2   & $\textbf{l}_2 = (13/200, 1/4800)$ & $Q_3 (\textbf{l}_2, x)$
                                                     & $1.0560478113$  & $0.4256711849$  \\
          5     &     &                       &      & $1.6388513157$  & $0.7163358746$  \\
          6     &     &                       &      & $2.3854005088$  & $0.7171809582$  \\
      \hline
          7     & 3   & $\textbf{l}_3 = (223758915/3305007602,$  
                                              &      & $3.1038729543$  & $0.8510463517$  \\
          8     &     & $ \quad  147/650416)$ & $M_4 (\textbf{l}_3, \textbf{r}_3, x)$ 
                                                     & $4.2595711727$  & $1.4178548432$  \\
          9     &     & $\textbf{r}_3 = (13/200, 1/4800)$          & 
                                                     & $5.7365650016$  & $1.4177054729$  \\
          10    &     &                       &      & $6.8904874142$  & $0.8442053143$  \\
      \hline
          11    & 4   & $\textbf{r}_4 = (0, 0)$ & $Q_3 (\textbf{r}_4, - x)$ refl.
                                                     & $7.5899955802$  & $0.6883344267$  \\
          12    &     &                       &      & $8.3059924679$  & $0.6859515449$  \\
          13    &     &                       &      & $8.8540119518$  & $0.3657140283$  \\
      \hline
    \end{tabular}
  \end{center}
\end{table}

\vfill

\bibliographystyle{elsarticle-num} 

\begin{thebibliography}{00}

\bibitem[Ait-Haddou \emph{et al.}(2015)]{AitBC2015}
R.~Ait-Haddou, M.~Bart\v{o}n, V.M.~Calo: \emph{Explicit Gaussian quadrature rules for cubic splines with non-uniform knot sequences}. J. Comput. Appl. Math. \textbf{290}, 543--552 (2015) 
\bibitem[Ait-Haddou \& Ruhland(2020)]{Ait2020}
R.~Ait-Haddou, H.~Ruhland: \emph{Asymptotically optimal quadrature rules for uniform splines over the real line}. Numerical Algorithms \textbf{86}, 1189--1223 (2021)
\href {https://doi.org/10.1007/s11075-020-00929-2} {\path{DOI :10.1007/s11075-020-00929-2}}.
\bibitem[Arvesu\ \emph{et al.}(2002)]{ArMaAl2002}
J.~Arves\'{u}, F.~Marcell\'{a}n, R.~\'{A}lvarez-Nodarse: \emph{On a Modification of the Jacobi Linear Functional: Asymptotic Properties and Zeros of the Corresponding Orthogonal Polynomials}. Acta Applicandae Mathematicae \textbf{71}, 127--158 (2002)
\bibitem[Bart\v{o}n \emph{et al.}(2017)]{BAC2017}
M.~Bart\v{o}n, R.~Ait-Haddou, V.M.~Calo: \emph{Gaussian Quadrature rules for $C^1$ quintic splines with
uniform knot vectors}. J. Comput. Appl. Math. \textbf{322}, 57--70 (2017) 
\bibitem[Bart\v{o}n \& Calo(2016)]{BC2016}
M.~Bart\v{o}n, , V.M.~Calo: \emph{Gaussian Quadrature for splines via homotopy continuation: rules for $C^2$ cubic splines}. J. Comput. Appl. Math. \textbf{296}, 709--723 (2016)
\bibitem[Chihara(1985)]{Chihara1985}
T.S.~Chihara: \emph{Orthogonal polynomials and measures with end point masses}. Rocky Mountain Journal of Mathematics, Volume \textbf{15}, Number 3, 705--719 Summer (1985) 
\bibitem[D\'{e}serti(2012)]{Des2012}
J.~D\'{e}serti: \emph{Some Properties of the Cremona Group}. Ensaios Matem{\'a}ticos \textbf{21}, 1--188 (2012) \href {https://doi.org/10.21711/217504322012/em211} {\path{DOI: 10.21711/217504322012/em211}}.
\bibitem[Dur\'{a}n \& de~la~Iglesia(2020)]{Dur2022}
A.J.~Dur\'{a}n, M.D.~de~la~Iglesia: \emph{Bispectral Jacobi type polynomials}. Advances in Applied Mathematics \textbf{136}, 102322 (2022) 
\href {https://doi.org/10.1016/j.aam.2022.102322} {\path{DOI: 10.1016/j.aam.2022.102322}}. 
\bibitem[Gautschi(2004)]{Ga2004}
W.~Gautschi: \emph{Orthogonal Polynomials: Computation and Approximation}. Numerical Mathematics and Scientific Computation, Oxford University Press, New York (2004)
\bibitem[Koornwinder(1984)]{Koor1984}
T.H.~Koornwinder: \emph{Orthogonal polynomials with weight function \\ $(1 - x)^{\alpha} (1 + x)^{\beta} + M \delta (x + 1) + N \delta (x - 1)$}. Canad. Math. Bull. \textbf{27}, 205--214 (1984)
\bibitem[Krall(1940)]{Krall1940}
H.L.~Krall: \emph{On orthogonal polynomials satisfying a certain fourth order differential equation}. The Pennsylvania State College Studies, No. \textbf{6}, 1--24 (1940) 
\bibitem[Nikolov(1996)]{Nik1996}
G.~Nikolov: \emph{On certain definite quadrature formulae}. J. Comput. Appl. Math. \textbf{75(2)}, 329--343 (1996)

\end{thebibliography}

\end{document}